\documentclass[10pt,twoside,a4paper,english,reqno]{amsart}
\usepackage[dvips]{epsfig}
\usepackage{lmodern}
\usepackage{amscd}
\usepackage{physics}
\usepackage{caption}
\usepackage{a4wide}
\usepackage{amssymb,amsthm,amsmath,amsfonts,mathrsfs}
\usepackage{yhmath}
\usepackage{bm,latexsym,enumitem,dsfont,tabularx,graphicx,url}
\usepackage[utf8]{inputenc}
\usepackage{hyperref}
\usepackage{stmaryrd}
\hypersetup{
    colorlinks=true,
    linkcolor=blue,
    filecolor=red,
    urlcolor=blue,
    citecolor=red,
}
\usepackage{nicefrac,microtype,upref,ulem,doi,footmisc}

\theoremstyle{plain}
\newtheorem{theorem}{Theorem}[section]

\newtheorem{lemma}[theorem]{Lemma}

\newtheorem{remark}[theorem]{Remark}

\numberwithin{equation}{section}

\newcommand{\Z}{\mathbb{Z}}
\newcommand{\N}{\mathbb N}

\numberwithin{figure}{section}
\numberwithin{table}{section}

\title[A FSG scheme for fCH equation]{A convergent Fourier spectral Galerkin method for the fractional Camassa-Holm equation}

\author{Mukul Dwivedi \and Andreas Rupp}
\address{Department of Mathematics, Saarland University, Saarbrücken, Germany}
\email{\{mukul.dwivedi;andreas.rupp\}@uni-saarland.de}

\begin{document}

% =============================================================
\begin{abstract}
We analyze a Fourier spectral Galerkin method for the fractional Camassa-Holm (fCH) equation involving a fractional Laplacian of exponent $\alpha \in [1,2]$ with periodic boundary conditions. The semi-discrete scheme preserves both mass and energy invariants of the fCH equation. For the fractional Benjamin-Bona-Mahony reduction, we establish existence and uniqueness of semi-discrete solutions and prove strong convergence to the unique solution in $ C^1([0,  T];H^{\alpha}_{\mathrm{per}}(I))$ for given $T>0$. For the general fCH equation, we demonstrate spectral accuracy in spatial discretization with optimal error estimates $\mathcal{O}(N^{-r})$ for initial data $u_0 \in H^r(I)$ with $r \geq \alpha + 2$ and exponential convergence $\mathcal{O}(e^{-cN})$ for smooth solutions.  Numerical experiments validate orbital stability of solitary waves achieving optimal convergence, confirming theoretical findings.
\end{abstract}

\maketitle
\section{Introduction}
The classical Camassa–Holm (CH) equation is a nonlinear dispersive partial differential equation originally derived to describe the unidirectional propagation of shallow water waves under gravity \cite{camassa1993integrable, fisher1999camassa, constantin2002stability}. In the shallow water regime, it models the balance between nonlinear advection and dispersion, and it provides a more faithful long-wave approximation than the standard Korteweg–de Vries (KdV) equation \cite{bona1975initial,sjoberg1970korteweg} in certain aspects. Notably, the CH equation can support solitary wave solutions (including peaked solitons) and allows for wave-breaking phenomena where wave slopes blow up in finite time \cite{camassa1993integrable,rodriguez2001cauchy}. To capture more complex dispersive behaviors observed in various media, one considers a fractional generalization of the CH equation that incorporates {\it nonlocal} dispersion \cite{mutlubacs2019cauchy}. In this fractional Camassa–Holm (fCH) model, the usual second-order spatial derivative is replaced by a fractional Laplacian operator $(-\Delta)^{\alpha/2}$, which accounts for long-range interactions and {\it anomalous dispersion} effects \cite{erbay2015derivation}. Moreover, the fCH framework often provides a more accurate description of wave dynamics than fractional KdV models \cite{dwivedi2024fully,dwivedi2024stability} because it retains the dual-balance structure between nonlinear advection and nonlocal dispersion, rather than being purely dispersive. This modification is physically motivated by situations such as wave propagation in nonlinear elastic media and complex fluids, where heterogeneities or memory effects give rise to dispersive characteristics that deviate from classical models. The resulting fCH equation extends the classical CH (recovered when $\alpha=2$) to a family of equations capable of describing wave dynamics in regimes where standard dispersive approximations (like KdV or the classical CH equation) become insufficient. Studying such fractional CH-type models is important both in theory and in practice: the nonlocal terms pose new analytical challenges in understanding well-posedness, stability, and wave-breaking behavior \cite{mao2021cauchy, fan2021cauchy, mutlubacs2019cauchy}, and they require advanced numerical techniques to simulate accurately.

In this paper, we consider the fractional Camassa-Holm (fCH) equation, for $\alpha \in [1,2]$, $T>0$, and $I:=[-\pi,\pi]$, given by
\begin{equation}\label{eq:fch}
\begin{aligned}
&u_t + f(u)_x + \mathcal{D}^\alpha u_t = -\kappa_2 \left[2 \mathcal{D}^\alpha (u u_x) + u \mathcal{D}^\alpha u_x \right], \quad \;\; x \in I, \;\; t \in (0,T],\\
  &u(x,0) = u_0(x), \quad x\in I,
\end{aligned}
\end{equation}
where $f(u) = \kappa_1u + \frac{3}{2}\gamma u^2$, $u_0$ is the given initial datum, and  \(u(x,t)\) denotes the $2\pi$-periodic real-valued fluid velocity. Here \(\kappa_1, \kappa_2 \in \mathbb{R}\) and \(\gamma > 0\) are physical parameters. The fractional Laplacian \(\mathcal{D}^\alpha := (-\Delta)^{\alpha/2}\) for \(\alpha > 0\) is defined via its Fourier series \cite{klein2024numerical}, i.e.,
\begin{equation}\label{eq:frac_def}
\widehat{\mathcal{D}^\alpha u}(k) = |k|^\alpha \hat{u}(k),
\end{equation}
where \(\hat{u}(k) = \frac{1}{2\pi} \int_I u(x) e^{-ikx}  dx\) is the \(k\)-th Fourier coefficient.

The formulation \eqref{eq:fch} generalizes the classical Camassa-Holm (CH) equation \cite{camassa1993integrable}
\begin{equation}\label{eq:ch}
u_t + 2\omega u_x + 3 u u_x - u_{xxt} = 2u_xu_{xx} + uu_{xxx}, \quad x \in I, \;\; t \in (0,T],
\end{equation}
subject to periodic boundary and initial conditions, which corresponds to the parameter choice \(\kappa_1 = 2\omega\), \(\gamma = 1\), \(\alpha = 2\), \(\kappa_2 = 1/3\). With \(\omega > 0\), the CH equation \eqref{eq:ch} describes the propagation of unidirectional gravitational waves in the shallow water approximation \cite{constantin2002stability,constantin1998global,fisher1999camassa}. The equation possesses remarkable mathematical properties including complete integrability \cite{fisher1999camassa}, finite-time wave breaking \cite{constantin1998global}, and peakon solutions \cite{rodriguez2001cauchy} of the form \(u(x,t) = ce^{-|x-ct|}\), for real constant $c$, when \(\omega = 0\). 

We also study a special case of \eqref{eq:fch} when \(\kappa_1 = 1\), \(\gamma = 1/3\), and \(\kappa_2 = 0\), which yields the $2\pi$-periodic fractional Benjamin-Bona-Mahony (fBBM) equation \cite{oruc2022stability}
\begin{equation}\label{eq:fbbm}
\begin{aligned}
&u_t + u_x + u u_x + \mathcal{D}^\alpha u_t = 0, \quad x \in I, \; t \in (0,T],\\
  &u(x,0) = u_0(x), \quad x\in I.
\end{aligned}
\end{equation}
For \(\alpha = 2\), the fBBM equation \eqref{eq:fbbm} reduces to the classical BBM equation \cite{medeiros1977existence}, an improvement over the Korteweg-de Vries equation \cite{bona1975initial} for modeling long surface gravity waves of small amplitude \cite{benjamin1972model}. The BBM equation balances nonlinear and dispersive effects in contexts including surface waves in liquids, hydromagnetic waves in plasma, and acoustic waves in crystals \cite{fabien2021high}. When \(\alpha = 1\), we obtain the regularized Benjamin-Ono (rBO) equation modeling long-crested waves at fluid interfaces, relevant to ocean pycnoclines and river-sea systems \cite{amaral2022existence,angulo2011regularized,oruc2022stability}.

Recent work by Klein and Oruc \cite{klein2024numerical} has extended the CH framework to fractional derivatives, where \(\alpha < 2\) corresponds to sub-dispersive regimes relevant to nonlinear elastic waves \cite{erbay2015derivation}. 
Despite theoretical advances in well-posedness \cite{mutlubacs2019cauchy,mao2021cauchy,fan2021cauchy}, numerical analysis of fractional CH equations remains underdeveloped, particularly for $\alpha < 2$. The classical CH equation has been extensively studied numerically, with various approaches developed including finite difference methods \cite{holden2006convergence,coclite2008convergent}, discontinuous Galerkin schemes \cite{xu2008local,liu2016invariant}, finite element methods \cite{antonopoulos2019error}, and variational discretizations \cite{galtung2021numerical}. Spectral methods have also been applied to the classical CH equation \eqref{eq:ch}, with Kalisch and Raynaud \cite{kalisch2006convergence} analyzing spectral projections and establishing $\mathcal{O}(N^{1-r})$ order of accuracy for initial data $u_0\in H^r_{\mathrm{per}}(I)$, which is not optimal, and Wang et al. \cite{wang2015generalized} develop generalized Laguerre spectral approximations. For the BBM equation, a  hybridizable discontinuous Galerkin scheme is presented in \cite{fabien2021high} and error analysis of a spectral projection of the rBO equation is given in \cite{kalisch2005error}.

For fractional extensions, numerical studies are notably scarce. Klein and Oruc \cite{klein2024numerical} conducted pioneering numerical investigations of fCH equations but without complete numerical analysis. For the fractional BBM equation \eqref{eq:fbbm}, analytical studies of solitary waves exist \cite{amaral2022existence,oruc2022stability}, but comprehensive numerical analysis remains lacking. This gap is particularly significant given the physical relevance of sub-dispersive regimes ($\alpha < 2$) to nonlinear elastic waves \cite{erbay2015derivation} and other applications where fractional operators model nonlocal interactions and anomalous dispersion.

Fourier spectral Galerkin methods present a natural approach to address these gaps, offering three key advantages for fractional equations. First, they exactly represent fractional operators through Fourier multipliers, avoiding approximation errors in the principal linear term. Second, they inherently preserve geometric structures and conservation laws when properly formulated. Third, they achieve spectral accuracy and provide exponential convergence superior to algebraic rates of finite difference or finite element methods \cite{canuto1988spectral,hesthaven2017numerical} for smooth solutions. These advantages make FSG methods particularly suitable for long-time simulations of nonlinear wave phenomena where conservation properties are crucial.

This work develops a comprehensive Fourier spectral Galerkin framework for both the general fCH equation \eqref{eq:fch} and its fBBM reduction \eqref{eq:fbbm}. We establish a semi-discrete scheme that preserves both mass and energy invariants of the continuous system. For the fractional BBM case, we prove the existence and uniqueness of solutions to the semi-discrete problem and demonstrate strong convergence to the unique solution of the continuous problem \eqref{eq:fbbm}. For the general fCH equation \eqref{eq:fch} with initial data $u_0\in H^r_{\mathrm{per}}(I)$, $r\geq \alpha$, we establish spectral accuracy in spatial discretization with optimal error estimates. Numerical investigations validate the theoretical framework, including orbital stability of solitary waves and finite-time cusp formation in sub-dispersive regimes. The method effectively handles both smooth solutions and non-smooth peakons, providing a versatile tool for exploring rich solution behaviors across fractional orders.

The paper is organized as follows: Section \ref{sec2} introduces function spaces, Fourier analysis, and fractional operators with key properties. Section \ref{sec3} develops the Fourier spectral Galerkin scheme and establishes conservation laws. Section \ref{sec4} provides convergence analysis, while Section \ref{sec5} presents numerical experiments validating theoretical results. Conclusions and future directions are discussed in Section \ref{sec6}.

\section{Preliminaries: Function Spaces and Fractional Operators}\label{sec2}
% Function spaces and approximation framework
This section introduces the function spaces and operators essential for our analysis. We consider $2\pi$-periodic functions defined on the interval $I = [-\pi, \pi]$. The space of square-integrable periodic functions, denoted by $L^2_{\mathrm{per}}(I)$, is equipped with the inner product
\[
(u, v) = \int_{-\pi}^{\pi} u(x) \overline{v(x)}  dx,
\]
and the induced norm $\|u\| = (u, u)^{1/2}$. For any real number $r \geq 0$, the periodic Sobolev space $H^r_{\mathrm{per}}(I)$ is defined as the set of periodic functions for which the norm
\[
\|u\|_r = \left( \sum_{k \in \mathbb{Z}} (1 + |k|^2)^r |\hat{u}(k)|^2 \right)^{1/2},
\]
is finite. Every function $\xi \in H^r_{\mathrm{per}}(I)$ admits a Fourier series expansion
\begin{equation}\label{eqn_fourier}
\xi(x) = \sum_{k = -\infty}^{\infty} \hat{\xi}(k) e^{ikx},
\end{equation}
which converges almost everywhere in $I$, see \cite{carleson1966convergence}. Note that $H^r_{\mathrm{per}}(I)$ is a subspace of the standard Sobolev space $H^r(I)$, and $L^2_{\mathrm{per}}(I) = H^0_{\mathrm{per}}(I)$ and the norm $\|\cdot\| =\|\cdot\|_0$.Z

For the purpose of spectral approximation, we define the finite-dimensional subspace of trigonometric polynomials
\[
V_N := \operatorname{span} \left\{ e^{ikx} : -N \leq k \leq N \right\}, \quad N \in \mathbb{N}.
\]
The spaces $V_N$ serve as natural approximations to $L^2_{\mathrm{per}}(I)$. The basis functions $\{ e^{i k x} \}_{k=-N}^{N}$ are orthogonal with respect to the inner product $(\cdot, \cdot)$. The orthogonal projection operator $P_N: L^2(I) \to V_N$ is defined by
\begin{equation}\label{eqn_projdefn}
P_N \xi (x) = \sum_{k = -N}^{N} \hat{\xi}(k) e^{i k x}.
\end{equation}
The orthogonality of the basis implies that $P_N$ satisfies the condition
\begin{equation}\label{eqn_ortho}
(P_N \xi - \xi, \phi) = 0 \quad \text{for all} \quad \phi \in V_N.
\end{equation}

We define the fractional Laplacian $\mathcal{D}^{\alpha}$ for $\alpha \geq 0$ using the Fourier series expansion \eqref{eqn_fourier} of a function $\xi \in H^r_{\mathrm{per}}(I),~r\geq 0,$ as follows
\begin{equation}\label{eqn_fracLap}
    \mathcal{D}^{\alpha}\xi(x) = \sum\limits_{k\in\Z}|k|^\alpha \hat{\xi}(k) e^{ikx}.
\end{equation}
Since $\phi \in V_N$, it has the representation:
\[
\phi(x) = \sum_{k=-N}^{N} c_k e^{ikx},
\]
where $c_k$ are coefficients. Moreover, using the orthogonality of $e^{ikx}$, the Fourier coefficients of $\phi$ are
% \hat{\phi}(m) = \frac{1}{2\pi} \int_{-\pi}^{\pi} \phi(x) e^{-imx}  dx = \frac{1}{2\pi} \int_{-\pi}^{\pi} \left( \sum_{k=-N}^{N} c_k e^{ikx} \right) e^{-imx}  dx, \quad m \in \mathbb{Z}.
% \] Therefore, using the orthogonality of $e^{ikx}$, we have
\[
\hat{\phi}(m) = 
\begin{cases} 
c_m & |m| \leq N, \\
0 & |m| > N.
\end{cases}
\]
Now, observe that, if  $ \phi\in V_N$, we get
    $$\mathcal D^\alpha \phi (x) =\sum\limits_{k\in\Z} |k|^{\alpha} \hat{\phi}(k) e^{ikx} = \sum\limits_{k=-N}^N |k|^{\alpha} c_k e^{ikx}.$$
Thus, $\mathcal D^\alpha \phi\in V_N$.  The following lemma states several important properties of the fractional Laplacian \eqref{eqn_fracLap}.
\begin{lemma}[See \cite{dwivedi2024numerical}]\label{lem:frac_properties}
    The fractional Laplacian \eqref{eqn_fracLap} satisfies the following properties
    \begin{enumerate}[label=\roman*)]
        \item For $\xi, \eta \in H^{\alpha}_{\mathrm{per}}(I)$, $\alpha \geq 0$, we have
        \begin{equation}\label{eqn_fracsymm}
            \left(\mathcal{D}^{\alpha}\xi, \eta\right) = \left(\xi, \mathcal{D}^{\alpha}\eta\right),
        \end{equation}
        and
        \begin{equation}\label{eqn_fracortho}
            \left(\mathcal{D}^{\alpha}\xi_x, \xi\right) = 0.
        \end{equation}

        \item  Let $\alpha_1, \alpha_2 \geq 0$, then for all $\xi, \eta \in H_{\mathrm{per}}^{\alpha_1+\alpha_2}(I)$, we have
        \begin{equation}\label{eqn_semigp}
            \left(\mathcal{D}^{\alpha_1+\alpha_2}\xi, \eta\right) = \left(\mathcal{D}^{\alpha_1}\xi, \mathcal{D}^{\alpha_2}\eta\right),
        \end{equation}
        and equivalently,
        \begin{equation}\label{eqn_semigp2}
            \mathcal{D}^{\alpha_1+\alpha_2}\xi = \mathcal{D}^{\alpha_1}\mathcal{D}^{\alpha_2}\xi = \mathcal{D}^{\alpha_2}\mathcal{D}^{\alpha_1}\xi.
        \end{equation}

        \item \label{compro}For an orthogonal projection $P_N$ defined by \eqref{eqn_ortho} and $\xi \in H^r_{\mathrm{per}}(I)$, $r \geq \alpha \geq 0$, the fractional Laplacian \eqref{eqn_fracLap} with exponent $\alpha$ commutes with $P_N$, i.e.,
        \begin{equation}\label{eqn_projcom}
            \mathcal{D}^{\alpha} (P_N \xi(x)) = P_N \mathcal{D}^{\alpha}\xi(x).
        \end{equation}
    \end{enumerate}
\end{lemma}

\begin{lemma}[See \cite{dwivedi2024numerical}]\label{lem:product_estimate}
    Let $\xi, \eta \in H^{\alpha}_{\mathrm{per}}(I)$, $\alpha \geq 0$. Then the following estimate holds
    \begin{equation}\label{eqn_lemma_comest}
        \mathcal{D}^\alpha(\xi\eta) \leq C(\alpha) \left(\xi \mathcal{D}^\alpha \eta + \eta \mathcal{D}^\alpha \xi\right),
    \end{equation}
    where $C(\alpha)$ is a constant depending on $\alpha$. Furthermore, there holds
    \begin{equation}\label{eqn_lemma_comest_morm}
        \norm{\mathcal{D}^\alpha(\xi\eta)} \leq C(\alpha) \left(\norm{\xi}_\infty \norm{\mathcal{D}^\alpha \eta} + \norm{\eta}_\infty \norm{\mathcal{D}^\alpha \xi}\right),
    \end{equation}
     {where $C(\alpha) = \max \left\{ 1, 2^{\alpha - 1} \right\}.$}
\end{lemma}

The convergence properties of the Fourier projection operator are fundamental to our error analysis. We first state a key lemma quantifying spectral approximation accuracy.
\begin{lemma}[Spectral Approximation; cf. Chapter 13 in  \cite{hesthaven2017numerical}]\label{Prop_Nr}
For any real \(r \geq s \geq 0\) and \(u \in H_{\mathrm{per}}^r(I)\), there exists \(C > 0\) independent of \(N\) such that
\[
\norm{u - P_N u}_s \leq C N^{s-r} \norm{u}_r.
\]
If \(u\) is analytic on \(I\), there exist positive constants \(C, \sigma\) independent of \(N\) such that
\[
\norm{u - P_N u} \leq C e^{-\sigma N} \norm{u}.
\]
\end{lemma}
While this result is standard \cite{hesthaven2017numerical}, we provide a proof for the specific case $\|\mathcal{D}^\alpha (u - P_N u)_x\|$ and $r\geq \alpha+2$ required in subsequent analysis.
\begin{proof}[Proof of Lemma \ref{Prop_Nr}] Using Parseval's identity, the projection property \eqref{eqn_ortho}, and the fact that $P_N$ commutes with derivatives and the fractional Laplacian from Lemma \ref{lem:frac_properties}, we have
\begin{align*}
\norm{\mathcal D^\alpha (u-P_Nu)_x}^2 &= \norm{\mathcal{D}^\alpha u_x}^2 - (\mathcal D^\alpha u_x, P_N \mathcal D^\alpha u_x) - (P_N \mathcal D^\alpha u_x, \mathcal D^\alpha u_x) + \norm{P_N \mathcal D^\alpha u_x}^2\\& = 2\pi \sum_{k \in \mathbb{Z}} k^2|k|^{2\alpha}|\hat{u}(k)|^2 - 2\pi \sum_{k=-N}^{N} k^2|k|^{2\alpha} |\hat{u}(k)|^2 \\
&= 2\pi \sum_{|k| > N} k^2|k|^{2\alpha}|\hat{u}(k)|^2 = 2\pi \sum_{|k| > N} \frac{|k|^{2r}}{|k|^{2r-2(\alpha+1)}}  |\hat{u}(k)|^2 \\
&\leq 2\pi N^{-2r+2(\alpha+1)}\sum_{|k| > N} |k|^{2r} |\hat{u}(k)|^2 \leq C N^{-2r+2(\alpha+1)} \norm{\partial_x^ru}^2\\&\leq C N^{-2r+2(\alpha+1)} \norm{u}_r^2.
\end{align*}
Therefore, we have
\begin{equation}\label{alpha_est}
    \norm{\mathcal D^\alpha (u-P_Nu)_x} \leq C N^{(\alpha+1)-r} \norm{u}_r.
\end{equation}
Furthermore, let $u$ be analytic on $I$, utilizing the above bound and Stirling's formula \cite[Theorem 13.3]{hesthaven2017numerical}, we have
\begin{align}\label{alpha_est2}
\norm{\mathcal D^\alpha (u-P_Nu)_x} \leq C N^{-r}N^{1+\alpha} \norm{\partial_x^ru} \leq C \frac{r!}{N^r}N^{1+\alpha} \norm{u} \leq C \frac{r^r e^{-r}}{N^r}N^{1+\alpha} \norm{u} \leq CN^{1+\alpha} e^{-cN} \norm{u},
\end{align}
where we have assumed that $r$ is proportional to $N$. This completes the proof.
\end{proof}

\begin{theorem}[See \cite{canuto1988spectral}]\label{proj_sob}
Let $u \in H_{\mathrm{per}}^s(I)$ for some $s> 0$. Then there exists a constant $C > 0$ independent of $N$ such that
\begin{equation}\label{sobolev-embedding}
\norm{u - P_N u}_{\infty} \leq C N^{-(s - 1/2)} \norm{\partial^s_x u}.
\end{equation}
\end{theorem}

\begin{theorem}\label{inv_sob}
Let $v_N \in V_N$. Then
\begin{equation}\label{bernstein-inequality}
\norm{v_N}_{\infty} \leq C N^{1/2} \norm{v_N},
\end{equation}
where $C$ is independent of $N$.
\end{theorem}

\begin{proof}
Express $v_N$ as
\[
v_N(x) = \sum_{k=-N}^{N} c_k e^{ikx}.
\]
The Cauchy-Schwarz inequality yields
\[
\abs{v_N(x)} \leq \sum_{k=-N}^{N} \abs{c_k} = \sum_{k=-N}^{N} 1 \cdot \abs{c_k} \leq \sqrt{\sum_{k=-N}^{N} 1^2} \cdot \sqrt{\sum_{k=-N}^{N} \abs{c_k}^2},
\]
which implies
\[
\abs{v_N(x)} \leq \sqrt{2N+1}\sqrt{\frac{1}{2\pi}} \norm{v_N} \leq CN^{1/2}\norm{v_N}.
\]
Thus
\[
\norm{v_N}_{\infty} \leq  CN^{1/2} \norm{v_N}.
\]
\end{proof}
Utilizing Theorem \ref{proj_sob} and \ref{inv_sob}, the following inequality holds.
For $u_N \in V_N$ and $u \in H_{\mathrm{per}}^s(I)$ for some $s > 0$, we have
\begin{equation}\label{inv_inq}
    \begin{aligned}
       \norm{u - u_N}_{\infty} &\leq \norm{u - P_N u}_{\infty} + \norm{P_N u - u_N}_{\infty}\\& \leq C N^{-(s - 1/2)} \norm{\partial_x^su} + C N^{1/2} \norm{P_N u - u_N} \\&
       \leq C N^{-(s - 1/2)} \norm{\partial_x^su} + C N^{1/2} \norm{ u - u_N}.
    \end{aligned}
\end{equation}

\section{Fourier Spectral Galerkin Scheme}\label{sec3}

The variational form of \eqref{eq:fch} is obtained by multiplying by a test function \(\phi \in H_{\mathrm{per}}^{\alpha}(I)\) and integrating over \(I\)
\begin{equation}\label{eq:weak_form}
(u_t, \phi) + ( u_t, \mathcal{D}^{\alpha}\phi) = (f(u),\phi_x) -\kappa_2 \left[ 2(u u_x, \mathcal{D}^{\alpha}\phi) + (u \mathcal{D}^\alpha u_x, \phi) \right].
\end{equation}

The Fourier spectral Galerkin approximation seeks solutions \(u_N(\cdot, t) \in V_N\) satisfying the variational formulation for all \(t > 0\) and test functions \(\phi \in V_N\):
\begin{multline}\label{eq:fsg}
\bigl((u_{N})_t, \phi\bigr) +\bigl( (u_{N})_t, \mathcal{D}^{\alpha}\phi\bigr) = \bigl(f(u_N),\phi_x\bigr) -\kappa_2 \left[ 2\bigl( u_N (u_N)_x, \mathcal{D}^{\alpha}\phi\bigr) + \bigl(u_N \mathcal{D}^\alpha (u_N)_x, \phi\bigr) \right].
\end{multline}
subject to the initial condition \(u_N(x, 0) = P_N u_0(x)\).  Expressing the numerical solution in Fourier series form \(u_N(x, t) = \sum_{k=-N}^{N} \hat{u}_k(t) e^{ikx}\) where $\hat{u}_k :=\hat u(k)$, and choosing \(\phi_m = e^{imx}\) for \(m = -N, \dots, N\) as test functions, utilizing \(\mathcal{D}^{\alpha} e^{imx} = |m|^\alpha e^{imx}\), we obtain
\begin{equation}\label{eq:ode_component}
\frac{d\hat{u}_m}{dt} = \frac{1}{1 + |m|^\alpha} \left( i\cdot m \widehat{f(u_N)}(m) - \kappa_2 \left[ 2 |m|^\alpha \widehat{u_N(u_N)_x}(m) + \widehat{u_N \mathcal{D}^\alpha (u_N)_x} (m) \right] \right),
\end{equation}
 Note that, \(\widehat{f(u_N)}(m)\), \(\widehat{u_N(u_N)_x}(m)\), and \(\widehat{u_N \mathcal{D}^\alpha (u_N)_x} (m)\) are nonlinear functions of the Fourier coefficients \(\hat{u}_k\). Specifically, we have
\begin{align*}
\widehat{u_N(u_N)_x}(m) &= \sum_{\substack{k+l=m \\ |k|,|l| \leq N}} (i \cdot l) \hat{u}_k \hat{u}_l, \\
\widehat{u_N \mathcal{D}^\alpha (u_N)_x} (m) &= \sum_{\substack{k+l=m \\ |k|,|l| \leq N}} |l|^\alpha (i\cdot l) \hat{u}_k \hat{u}_l, \\
\widehat{f(u_N)}(m) &= \kappa_1 \hat{u}_m + \frac{3}{2} \gamma \sum_{\substack{k+l=m \\ |k|,|l| \leq N}} \hat{u}_k \hat{u}_l.
\end{align*}
Defining the coefficient vector \(\mathbf{U}(t) = [\hat{u}_{-N}(t), \dots, \hat{u}_N(t)]^T \in \mathbb{R}^{2N+1}\), the system \eqref{eq:ode_component} for \(m = -N, \dots, N\) constitutes the semi-discrete ODE system
\begin{equation}\label{eq:galerkin_ode}
\frac{d\mathbf{U}}{dt}(t) = \mathbf{g}(\mathbf{U}(t)), \quad \mathbf{g} = [g_{-N}, \dots, g_N]^T,
\end{equation}
where each \(g_m: \mathbb{R}^{2N+1} \to \mathbb{R}\), for $m = -N,\cdots,N$, are smooth and locally Lipschitz continuous, as the right-hand side involves finite sums of products of Fourier coefficients $\hat u_\ell$ and bounded multipliers \(|l|^\alpha\) for \(|l| \leq N\). Initial conditions are specified via the \(L^2\)-orthogonal projection \(u_N(x, 0) = P_N u_0(x)\) that follows
\[
\mathbf{U}(0)  = \bigl\{\hat{u}_{k}(0)\bigr\}_{k=-N}^N  = \bigl\{ (u_0, e^{iNx} )\bigr\}_{k=-N}^N,
\]
By the Picard-Lindel\"of theorem, there exists \(T_{\max} > 0\) such that \eqref{eq:galerkin_ode} admits a unique solution for \(t \in [0, T_{\max})\). As established in Theorem \ref{thm:conservation}, the solution remains uniformly bounded for all \(t > 0\), permitting global unique extension to \([0, \infty)\).

The semi-discrete scheme preserves the geometric invariants of the continuous equation.
\begin{theorem}[Conservation]\label{thm:conservation}
Solutions to the semi-discrete Fourier spectral Galerkin scheme \eqref{eq:fsg} satisfy:
\begin{align}
\frac{d}{dt} \int_I u_N  \, dx &= 0, \label{con0}\\
% \frac{d}{dt} \int_I \left( u_N + \mathcal{D}^\alpha u_N \right) \, dx &= 0, \label{con1}\\
\frac{d}{dt} \int_I \left( u_N^2 + |\mathcal D^{\alpha/2} u_N|^2 \right) \, dx &= 0 \label{cons2}.
\end{align}
\end{theorem}

\begin{proof} 
Choose \(\phi \equiv 1\) in \eqref{eq:fsg}, then using the property  \eqref{eqn_fracortho} of the fractional Laplacian and $\mathcal D^\alpha \phi =0$ yields
\[
\frac{d}{dt} \int_I  u_N \,  dx = -\kappa_2  \int_I u_N \mathcal{D}^\alpha (u_N)_x \,dx  =0,
\]
which proves \eqref{con0}.
% Now, first use property \eqref{eqn_fracsymm} in \eqref{eq:fsg}, then choose $\phi \equiv 1$, and eventually using the property \eqref{eqn_fracortho} and periodicity of $u_N$, we have 
% \[
% \frac{d}{dt} \int_I \left( u_N + \mathcal{D}^\alpha u_N \right)  dx = -\kappa_2 \left[ \int_I (\mathcal{D}^{\alpha}u_N^2)_x \,dx  + \int_I u_N \mathcal{D}^\alpha (u_N)_x \,dx \right] =0,
% \]
% which proves \eqref{con1}.
Next, choose \(\phi = u_N \) in \eqref{eq:fsg}, and using the fractional Laplacian property \eqref{eqn_fracsymm}, we get 
\[
 \frac{1}{2} \frac{d}{dt} \bigl( \|u_N\|^2 + \|\mathcal{D}^{\alpha/2} u_N\|^2 \bigr) = \bigl(f(u_N), (u_N)_x\bigr) - \kappa_2 \left[2\bigl( \mathcal{D}^{\alpha} (u_{N} (u_{N})_{x}), u_N \bigr) + \bigl( u_{N} \mathcal{D}^{\alpha} (u_{N})_{x}, u_N \bigr)\right].
\]
 In the right-hand side of the above equation, we employ integration by parts, periodicity of $u_N$, and property \eqref{eqn_fracsymm}, which implies
\begin{align*}
\bigl(f(u_N), (u_N)_x\bigr) &= \int_I  \left( -\frac{\kappa_1}{2} u_N^2 - \frac{1}{2} u_N^3 \right)_x dx = 0, \\
2\kappa_2 \bigl( \mathcal{D}^{\alpha} (u_{N} (u_{N})_{x}), u_N \bigr) + \kappa_2 \bigl( u_{N} \mathcal{D}^{\alpha} (u_{N})_{x}, u_N \bigr) &= -\kappa_2 \bigl(u_{N}^2, \mathcal{D}^{\alpha}(u_N)_x \bigr) + \kappa_2 \bigl( \mathcal{D}^{\alpha} (u_{N})_{x}, u_N^2 \bigr) = 0.
\end{align*}
Thus
\[
 \frac{d}{dt} \int_I \left( u_N^2 + |\mathcal D^{\alpha/2} u_N|^2 \right) dx = 0.
\]
\end{proof}

\subsection{Convergence analysis of the FSG scheme}

The convergence analysis of the general fCH equation \eqref{eq:fch} presents substantial challenges due to the complex nonlinear coupling between the solution $u$ and the fractional operator $\mathcal{D}^\alpha$ in the terms $\mathcal{D}^\alpha(u u_x)$ and $u \mathcal{D}^\alpha u_x$. These nonlinearities introduce analytical difficulties in establishing full well-posedness and convergence results for the spectral method. To overcome these obstacles while retaining essential fractional dynamics, we consider the special case $\kappa_2 =0$, the fBBM equation \eqref{eq:fbbm}. This simplified model eliminates the problematic nonlinear coupling terms while preserving the fractional dissipation through $\mathcal{D}^\alpha u_t$. The resulting structure allows us to establish full convergence results; we prove that the spectral solution $u_N$ converges to the unique solution of the fBBM equation \eqref{eq:fbbm} using compactness arguments.

\begin{theorem}[Well-posedness]\label{thm:wellposed}
For initial data \(u_{0}\in H_{\mathrm{per}}^{r}(I)\) with \(r\geq \alpha\) and \(T>0\) is finite, there exists a unique solution
\(
u\in C^{1}\!\bigl([0,T];H_{\mathrm{per}}^{r}(I)\bigr)
\)
to \eqref{eq:fbbm}.
\end{theorem}

The constructive proof of Theorem~\ref{thm:wellposed} follows from the next lemma together with the extension argument formulated in Remark~\ref{remk_extnsn}.

\begin{lemma}\label{conv_lemma}
  Let $u_N$ denote the approximate solution to the fBBM equation \eqref{eq:fbbm}  obtained via the semi-discrete FSG scheme \eqref{eq:fsg} with parameters $(\kappa_1, \gamma, \kappa_2) =(1, 1/3,0)$. For initial data $u_0 \in H^\alpha_{\mathrm{per}}(I)$, there exists a time $\bar T$ and a positive constant $C = C(\|u_0\|_\alpha)$ such that the following bounds hold
    \begin{align}
        \label{bd1} \|u_N(t)\| &\leq C, \\
        \label{bd2} \|u_N(t)\|_{\alpha} &\leq C, \\
        \label{bd3} \|(u_N)_t(t)\|_{\alpha} &\leq C,
    \end{align}
    for all $t \in [0, \bar T]$.
    Furthermore, as $N \to \infty$, the sequence $\{u_N\}$ converges to the unique solution of the fBBM equation \eqref{eq:fbbm} in the topology $ C^1([0, \bar T];H^{\alpha}_{\mathrm{per}}(I)).$ 
\end{lemma}

\begin{proof} 
The bound \eqref{bd1} follows from the Theorem \eqref{thm:conservation} for all $t>0$.  
To bound \(\norm{u_N(t)}_{\alpha}\), consider the FSG scheme \eqref{eq:fsg} with parameters  $(\kappa_1, \gamma, \kappa_2) =(1, 1/3,0)$ and choose test function \(\phi = \mathcal{D}^{\alpha} u_N\). We have
\begin{align*}
\left( (u_N)_t, \mathcal{D}^{\alpha} u_N \right) + \left( \mathcal{D}^{\alpha} (u_N)_t, \mathcal{D}^{\alpha} u_N \right) = \left( f(u_N), (\mathcal{D}^{\alpha} u_N)_x \right) =\left( u_N, \mathcal{D}^{\alpha} (u_N)_x \right) +\frac{1}{2}  \left( (u_N)^2, \mathcal{D}^{\alpha} u_N \right).
\end{align*}
Since $u_N$ is periodic, using integration by parts, property \eqref{eqn_fracortho}, and the Cauchy-Schwarz inequality,  we have
\begin{align*}
\frac{1}{2}\frac{d}{dt}\bigl(\|\mathcal{D}^{\frac{\alpha}{2}} u_N\|^2+\|\mathcal{D}^{\alpha} u_N\|^2\bigr)=  -  \left( u_N (u_N)_x, \mathcal{D}^{\alpha} u_N \right)  \leq    \|u_N\|_{\infty} \|(u_N)_x\|\| \mathcal{D}^{\alpha} u_N\|.
\end{align*}
Employing Sobolev inequalities $\|u_N\|_{\infty}\leq C\|u_N\|_{\alpha}$  and  $\|(u_N)_x\|\leq C\|\mathcal D^\alpha u_N\|\leq C\|u_N\|_{\alpha}$ from \cite[Lemma A.1]{dwivedi2024numerical}, Young's inequality, and then adding some positive quantity to the right-hand side, we have
\begin{equation}
\begin{aligned}
\frac{1}{2}\frac{d}{dt}\bigl(\|\mathcal{D}^{\frac{\alpha}{2}} u_N\|^2+\|\mathcal{D}^{\alpha} u_N\|^2\bigr)&\leq   C \|u_N\|_{\alpha}^2\| \mathcal{D}^{\alpha} u_N\| \leq  C\|u_N\|_{\alpha}^4 + C \| \mathcal{D}^{\alpha} u_N\|^2 \\&  \leq C\bigl(\| u_N\|_{\frac{\alpha}{2}}^2+\|u_N\|_{\alpha}^2\bigr) + C\bigl(\| u_N\|_{\frac{\alpha}{2}}^2+\|u_N\|_{\alpha}^2\bigr)^2\label{eqn_bddp2}  .
\end{aligned}
\end{equation}
Applying analogous energy estimates to the lower-order derivative terms in the full Sobolev norm and combining with \eqref{eqn_bddp2} yields
\begin{align}\label{eqn:de1}
\frac{1}{2}\frac{d}{dt}\bigl(\| u_N\|_{\frac{\alpha}{2}}^2+\| u_N\|_{\alpha}^2\bigr)\leq  C( \|u_N\|_{\frac{\alpha}{2}}^2+\| u_N\|_{\alpha}^2 ) + C( \|u_N\|_{\frac{\alpha}{2}}^2+\| u_N\|_{\alpha}^2 )^2.
\end{align}
Define the functional $E(t) := \|u_N(t)\|_{\frac{\alpha}{2}}^2 + \|u_N(t)\|_{\alpha}^2$. From the differential inequality \eqref{eqn:de1}, we have
\begin{align}
\frac{dE}{dt} \leq 2C E + 2C E^2 = 2C E(1 + E),
\end{align}
with initial condition \(E(0) = E_0 := \|u_0\|_{\frac{\alpha}{2}}^2 + \|u_0\|_{\alpha}^2\). Consider the differential equation
\begin{equation}
\frac{dz}{dt} = 2C z(1 + z), \quad z(0) = E_0,
\end{equation}
where C is independent of $t$. The solution $z(t) = (E_0 e^{2C t})/(1 + E_0 - E_0 e^{2C t})$ is increasing and convex for all time  \(t < \frac{1}{2C} \ln(1 + E_0^{-1})=:T_{b}\). Define $
\bar T := \frac{T_b}{2} $. For $t\leq \bar T$, we have 
\begin{align}
e^{2C t} \leq e^{2C \bar T} = \left(1 + \frac{1}{E_0}\right)^{1/2},
\end{align}
which yields the bound
\begin{align}
z(t) \leq z(\bar T) = \frac{E_0 (1 + E_0^{-1})^{1/2}}{1 + E_0 - E_0 (1 + E_0^{-1})^{1/2}} \leq 2(1 + E_0).
\end{align}
Now consider the function $w(t) = z(t) - E(t)$. Then
\begin{align*}
\frac{dw}{dt} &= \frac{dz}{dt} - \frac{dE}{dt} \geq 2C z(1 + z) - 2C E(1 + E)= 2C \left[ (z - E) + (z^2 - E^2) \right]  = 2C w \left(1 + z + E \right).
\end{align*}
Since $z(t) > 0$ and $E(t) \geq 0$ for $t < \frac{1}{2C} \ln(1 + E_0^{-1})$, we have $1 + z + E > 0$. With the initial condition $w(0) = z(0) - E(0) = 0$, the differential inequality $\frac{dw}{dt} \geq K(t) w$ (where $K(t) = 2C(1 + z + E) > 0$) implies $w(t) \geq 0$ for $t \in [0, \bar T]$ by the following argument. 
Assume, for the sake of contradiction, that there exists some \( t_1 > 0 \) such that \( w(t_1) < 0 \). Since \( w(0) = 0 \) and \( w \) is continuous, let \( t_0 = \inf \{ t > 0 : w(t) < 0 \} \). By continuity, \( w(t_0) = 0 \), and there exists \( \delta > 0 \) such that \( w(t) < 0 \) for all \( t \in (t_0, t_0 + \delta) \).

For \( t \in (t_0,t_0+\delta) \), since \( w(t) < 0 \) and \( K(t) = 2C(1 + z + E) > 0 \), the inequality \( \frac{dw}{dt} \geq K(t) w \) holds. Set \( y(t) = -w(t) > 0 \). Then \( \frac{dw}{dt} = -\frac{dy}{dt} \), and substituting gives
\(
\frac{dy}{dt} \leq K(t) y.
\)
Define \( v(t) = y(t) \exp\left( -\int_{t_0}^t K(s)  ds \right) \). Then
\[
\frac{dv}{dt} = \left( \frac{dy}{dt} - K(t) y \right) \exp\left( -\int_{t_0}^t K(s)  ds \right) \leq 0,
\]
since \( \frac{dy}{dt} \leq K(t) y \). Thus, \( v \) is non-increasing. At \( t = t_0 \), \( y(t_0) = -w(t_0) = 0 \), so \( v(t_0) = 0 \). Therefore, for \( t > t_0 \), \( v(t) \leq v(t_0) = 0 \). However, since \( y(t) > 0 \) and \( \exp\left( -\int_{t_0}^t K(s)  ds \right) > 0 \) for \( t > t_0 \), it follows that \( v(t) > 0 \). This contradiction implies that \( w(t) \geq 0 \) for all \( t \in [0, \bar{T}] \).

Thus $ E(t)\leq z(t)\leq 2(1+E_0)$ on $[0, \bar T]$.
Therefore, we have
\begin{equation}
    \| u_N(t)\|_{\frac{\alpha}{2}}^2+\| u_N(t)\|_{\alpha}^2\leq C(\| u_0\|_{\alpha}), 
\end{equation}
for all $t\leq \bar T$. This proves the bound \eqref{bd2}. Now consider the FSG scheme \eqref{eq:fsg} again with parameters $(\kappa_1, \gamma, \kappa_2) =(1, 1/3,0)$, and choose test function $\phi = (\mathcal{D}^{\alpha}u_N)_t$ that yields
\begin{align*}
\left( (u_N)_t, (\mathcal{D}^{\alpha}u_N)_t \right) + \left( \mathcal{D}^{\alpha} (\mathcal{D}^{\alpha}u_N)_t, (u_N)_t \right) = \left( f(u_N), (\mathcal{D}^{\alpha}u_N)_{tx} \right).
\end{align*}
Once again by the integration by parts and using the periodicity $u_N$, the Cauchy-Schwarz inequality, and the Sobolev inequalities \cite[Lemma A.1]{dwivedi2024numerical}, we have
\begin{align*}
\| \mathcal{D}^{\frac{\alpha}{2}}(u_N)_t\|^2+\| \mathcal{D}^{\alpha}(u_N)_t\|^2&= - \left( (u_N)_x, (\mathcal{D}^{\alpha}u_N)_t\right) -  \left( u_N (u_N)_x, \mathcal{D}^{\alpha}(u_N)_t \right) \\& \leq   \|(u_N)_x\| \|\mathcal{D}^{\alpha}(u_N)_t\| + \|u_N\|_{\infty} \|(u_N)_x\|\| \mathcal{D}^{\alpha}(u_N)_t\|\\& \leq \|\mathcal D^{\alpha}u_N\| \|\mathcal{D}^{\alpha}(u_N)_t\| + \|u_N\|_{\alpha}^2\| \mathcal{D}^{\alpha}(u_N)_t\| \\& 
\leq C_\delta\|\mathcal D^{\alpha}u_N\|^2  + \delta \|\mathcal{D}^{\alpha}(u_N)_t\|^2 + C_\delta\|u_N\|_{\alpha}^4 + \delta \| \mathcal{D}^{\alpha}(u_N)_t\|^2,
\end{align*}
where we have used Young's inequality and $\delta>0$ is a small constant. Eventually, using the bound \eqref{bd2}, we have  
\begin{align*}
\| (\mathcal{D}^{\frac{\alpha}{2}}u_N)_t\|^2+\| (\mathcal{D}^{\alpha}u_N)_t\|^2 \leq C,
\end{align*}
for all $t\leq \bar T$. Similarly, we get the estimate for the lower-order derivatives. This proves the bound \eqref{bd3}.

The bound \eqref{bd3} implies that $u_{N} \in \text{Lip}([0, \bar T];H^{\alpha}_{\mathrm{per}}(I))$ for all $N$. Then using \eqref{bd1}, we employ an application of the Arzelà-Ascoli theorem, which ensures the sequential compactness of $\{u_{N}\}_{N\in\N}$ in $C([0, \bar T];H^{\alpha}_{\mathrm{per}}(I))$. Consequently, there exists a subsequence $N_k$ such that
    \begin{equation}\label{unifcon}
        u_{N_k} \to \bar u \text{ uniformly in } C([0, \bar T];H^{\alpha}_{\mathrm{per}}(I)) \text{ as } N_k \to \infty.
    \end{equation}
    We now show that $\bar u$ is the unique solution of the fBBM equation \eqref{eq:fbbm}. First, claim that $\bar u$ satisfies the following 
    \begin{equation}
\begin{aligned}
 \int_0^{\bar T} \int_I (\bar u +\mathcal{D}^\alpha \bar u) \varphi_t \,dx\,dt + \int_I (u_0+\mathcal{D}^\alpha u_0)\varphi(\cdot, 0)\,dx = \int_0^{\bar T}\int_I f(\bar u) \varphi_x \,dx\,dt,
\end{aligned}
\end{equation}
for all $\phi \in C_c^\infty(I \times  [0,\bar T])$. Consider a function $\varphi \in C_c^\infty(I \times  [0,\bar T])$. Using the projection $P_N \varphi(\cdot, t)$ as the test function in the FSG scheme \eqref{eq:fsg} and integrating over $[0,\bar T]$, we get
\begin{equation}\label{eqn:prjtest}
    \int_0^{\bar T}\int_I ((u_N)_t +\mathcal{D}^\alpha (u_N)_t) P_N \varphi\, dt = \int_0^{\bar T} \int_I f(u_N) (P_N \varphi)_x \,dt.
\end{equation}
Integrating by parts in time for the left-hand side gives
\begin{align*}
\int_0^{\bar T} \int_I(u_N)_t \varphi \,dx\, dt &= - \int_0^{\bar T} \int_I u_N \varphi_t \,dx\,dt - \int_I u_N(\cdot, 0)\varphi(\cdot, 0)\,dx, \\
\int_0^{\bar T} \int_I \mathcal{D}^\alpha (u_N)_t \varphi\,dx\, dt &= - \int_0^{\bar T}\int_I \mathcal{D}^\alpha u_N \varphi_t\,dx\, dt - \int_I\mathcal{D}^\alpha u_N(\cdot, 0) \varphi(\cdot, 0)\,dx.
\end{align*}
For the right-hand side, we have
\begin{equation*}
\begin{aligned}
\int_0^{\bar T} \int_I f(u_N)  (P_N \varphi)_x \,dx\,dt = \int_0^{\bar T}\int_I f(u_N) \varphi_x \,dx\,dt - \int_0^{\bar T}\int_I f(u_N)_x  (P_N \varphi - \varphi)\, dx\,dt.
\end{aligned}
\end{equation*}
Now substituting the above identities in \eqref{eqn:prjtest} implies
\begin{equation}
\begin{aligned}
 \int_0^{\bar T} \int_I (u_N +\mathcal{D}^\alpha u_N) \varphi_t \,dx\,dt &+ \int_I (u_N(\cdot, 0)+\mathcal{D}^\alpha u_N(\cdot, 0))\varphi(\cdot, 0)\,dx\\& = -\int_0^{\bar T}\int_I f(u_N) \varphi_x \,dx\,dt + \int_0^{\bar T}\int_I f(u_N)_x  (P_N \varphi - \varphi)\, dx\,dt.
\end{aligned}
\end{equation}
Since $P_N\varphi \to \varphi$ and $u_{N} \to \bar{u}$ in $C([0, \bar T]; H^\alpha_{\mathrm{per}}(I))$, passing to the limit as $N \to \infty$ yields
\begin{equation}
\begin{aligned}
 \int_0^{\bar T} \int_I (\bar u +\mathcal{D}^\alpha \bar u) \varphi_t \,dx\,dt + \int_I (u_0+\mathcal{D}^\alpha u_0)\varphi(\cdot, 0)\,dx = \int_0^{\bar T}\int_I f(\bar u) \varphi_x \,dx\,dt,
\end{aligned}
\end{equation}
for all $\phi \in C_c^\infty(I \times  [0, \bar T])$, where initial conditions converge as $u_N(\cdot, 0) = P_N u_0 \to u_0$ and $\mathcal{D}^\alpha u_N(\cdot, 0) \to \mathcal{D}^\alpha u_0$. Finally, the regularity $\bar{u} \in  C^1([0,\bar T]; H^{\alpha}_{\mathrm{per}}(I))$ implies that $\bar u$ is actually a strong solution and satisfies the fBBM equation \eqref{eq:fbbm} as an $L^2$-identity.

For uniqueness, let $\bar{u}$ and $\bar{w}$ be solutions of the fBBM equation \eqref{eq:fbbm} with $\bar{u}(x, 0) = \bar{w}(x, 0) = u_0(x)$. Define $v = \bar{u} - \bar{w}$, which satisfies $v(x, 0) = 0$, periodic condition, and
\[
(1 + \mathcal{D}^\alpha) v_t = -(v_x + (\bar u\bar u_x- \bar w\bar w_x)).
\]
Taking the inner product with $\mathcal D^\alpha v$ and using Lemma \ref{lem:frac_properties}, we have
\begin{align*}
\frac{1}{2} \frac{d}{dt} \left( \|\mathcal{D}^{\frac{\alpha}{2}} v\|^2 + \|\mathcal{D}^{\alpha} v\|^2 \right)&= -(v_x + (\bar u\bar u_x- \bar w\bar w_x),\mathcal D^\alpha v)\\& = -(v\bar u_x+\bar wv_x, \mathcal D^\alpha v)\\& = -(\bar u_x,v\mathcal D^\alpha v) - (\bar w,v_x \mathcal D^\alpha v).
\end{align*}
The Cauchy-Schwarz inequality, Sobolev embeddings $\|v\|_{\infty}\leq C\| v\|_{\alpha} \leq C(\| \bar u\|_{\alpha} +\| \bar w\|_{\alpha}) $ and $\|v_x\|\leq C\| \mathcal D^\alpha v\|$ from \cite[Lemma A.1]{dwivedi2024numerical}, and Young's inequality imply
\[
\left| (\bar u_x,v\mathcal D^\alpha v) + (\bar w,v_x \mathcal D^\alpha v) \right| \leq K (\|v\|_{\infty}\|\mathcal D^\alpha v\| + \|v_x\|\|\mathcal D^\alpha v\|)\leq KC \left( \|\mathcal D^\alpha v\|^2 + \|\mathcal{D}^{\frac{\alpha}{2}} v\|^2 \right),
\]
where $K = C \sup_{t \in [0,\bar T]} (\|\bar{u}(t)\|_{\alpha} + \|\bar{w}(t)\|_{\alpha})$. The Gr\"onwall's inequality with $v(0) = 0$ yields $v = 0$. Thus, $\bar{u}$ is unique, and the entire sequence $u_N$ converges to $\bar{u}$ in $C^1([0, \bar T]; H^{\alpha}_{\mathrm{per}}(I))$.
\end{proof}

\begin{remark}\label{remk_extnsn}
Sj\"oberg’s technique \cite{sjoberg1970korteweg} can be adapted to extend the local existence result for the fBBM equation to any finite time interval.  The crucial point is that the length of the existence interval \([0,\bar{T}]\) depends only on \(\|u_{0}\|_{H_{\mathrm{per}}^{\alpha}}\). Since the $H_{\mathrm{per}}^{\alpha}$-norm of the exact fBBM solution is uniformly bounded in time \cite{amaral2022existence}, we can extend the solution iteratively, and one may piece together successive local solutions to obtain a global solution.  More precisely, at time \(\bar{T}\) we define a new initial condition by projection,
\[
u_{\bar{T}}:=P_{N}u(\cdot,\bar{T}),
\]
and solve the fBBM equation with data \(u_{\bar{T}}\) on the interval \([\bar{T},2\bar{T}]\).  Since each such step preserves the \(H_{\mathrm{per}}^{\alpha}\)-norm, repeating this procedure inductively yields a solution on \([0,T]\) for a given finite \(T>0\). Also see the approach given in \cite{constantin1998wave} for finite time wave breaking for similar nonlinear nonlocal shallow water equations.
\end{remark}

\section{Optimal error analysis}\label{sec4}

This section establishes the spectral convergence of the FSG scheme \eqref{eq:fsg} for the fCH equation \eqref{eq:fch}. To derive optimal spectral convergence rates, we introduce an \textit{a priori} assumption on the approximation error and subsequently justify its validity through a continuity argument.

Assume the exact solution $u$ of \eqref{eq:fch} possesses sufficient regularity such that $u \in C([0,T]; H^r(I))$ with $r \geq \alpha + 2$. We assume the following error bound for the semi-discrete approximation:
\begin{equation}\label{prioriass}
    \|u - u_N\| \leq CN^{-(\alpha + 3/2)}.
\end{equation}
A direct consequence of inequality \eqref{inv_inq} and the approximation property \eqref{Prop_Nr} yields 
\begin{equation}\label{ass1uu_h2}
    \|u - u_N\|_{\infty} \leq CN^{-(\alpha + 1)}.
\end{equation}
The admissibility of \eqref{prioriass} is established in the Remark \ref{Ass_vald}. Define the projection error $\mathcal{E}_N(t) \in V_N$ as  
\begin{equation*}
    \mathcal{E}_N(t) := P_N u(t) - u_N(t), \quad \mathcal{E}_N(0) = P_N u_0 - u_N(0) = 0.
\end{equation*}
This quantity represents the discrepancy between the spectral projection of the exact solution and the numerical approximation.

\begin{theorem}[Spectral convergence]\label{thm:spectral_accuracy}
Let $u$ denote the exact solution of \eqref{eq:fch} satisfying $u \in C([0,T]; H^r(I))$ with regularity index $r \geq \alpha + 2$, and let $u_N$ be the semi-discrete approximation obtained via \eqref{eq:fsg}. There exists a constant $C = C(T, r, \|u\|_{C([0,T];H^r)})$, independent of the discretization parameter $N$, such that  
\begin{equation}\label{eq:L2_error_bound}
  \|u(t) - u_N(t)\| \leq CN^{-r},
\end{equation}
for all $t<T$. If, in addition, the initial data \(u_0\) is real-analytic, then there exist positive constants \(C\) and \(c\), independent of \(N\), such that
\[
\|u(t)-u_N(t)\|_{L^2}\le C\,\mathrm{e}^{-cN},
\]
for all $t<T$.
\end{theorem}

\begin{proof}
The error evolution equation is derived by choosing the test function $\phi = \mathcal{E}_N$ in the weak formulation \eqref{eq:fsg} and the analysis leverages the commutativity of $\mathcal{D}^\alpha$ with $P_N$ (Lemma \ref{lem:frac_properties}). This yields
\begin{equation}\label{eq:est_phi}
    \begin{aligned}
        \frac{1}{2}\frac{d}{dt}\left(\|\mathcal{E}_{N}\|^2 + \|\mathcal{D}^{\frac{\alpha}{2}}\mathcal{E}_N\|^2\right)&= (P_Nu_t,\mathcal E_N) + (P_N(\mathcal{D}^{\frac{\alpha}{2}}u)_t, \mathcal{D}^{\frac{\alpha}{2}}\mathcal E_N) - \bigl((u_N)_t,\mathcal E_N\bigl)- \bigl((\mathcal{D}^{\frac{\alpha}{2}}u_N)_t, \mathcal{D}^{\frac{\alpha}{2}}\mathcal E_N\bigr)\\
        &=\bigl(f(u)-f(u_N),(\mathcal E_N)_x\bigr) -\kappa_2 \Bigl[ \bigl(\mathcal{D}^{\frac{\alpha}{2}} u^2_x-\mathcal{D}^{\frac{\alpha}{2}} (u^2_N)_x, \mathcal{D}^{\frac{\alpha}{2}}\mathcal E_N\bigr) \\&\qquad + \bigl(u \mathcal{D}^\alpha u_x-u_N \mathcal{D}^\alpha (u_N)_x, \mathcal E_N\bigr) \Bigr]\\ 
        &=\underbrace{\bigl(f'(u)(\mathcal E_N-(P_Nu-u)) + \tilde f''_u(u-U_N)^2,(\mathcal E_N)_x\bigr)}_{\mathcal I_1}\\&\qquad \underbrace{-\kappa_2 \Bigl[ \bigl(\mathcal{D}^{\frac{\alpha}{2}}(2u(\mathcal E_N-(P_Nu-u))-(u-u_N)^2)_x, \mathcal{D}^{\frac{\alpha}{2}}\mathcal E_N\bigr)}_{\mathcal I_2} \\&\qquad + \underbrace{\bigl((u-u_N)\mathcal D^\alpha u_x - (u-u_N)\mathcal{D}^\alpha(u-u_N)_x + u\mathcal{D}^\alpha(u-u_N)_x, \mathcal E_N\bigr)}_{I_3} \Bigr],
    \end{aligned}
\end{equation}
where the right-hand side is decomposed using the following Taylor expansion and algebraic identities
\begin{equation}
    \begin{aligned}
        f(u) - f(U_N) &= f'(u)(u-U_N) + \tilde f''_u(u-U_N)^2 = f'(u)(\mathcal E_N-(P_Nu-u)) + \tilde f''_u(u-U_N)^2,\\
        u^2 - u^2_N &= 2u(u-u_N)-(u-u_N)^2 = 2u(\mathcal E_N-(P_Nu-u))-(u-u_N)^2,\\
        u \mathcal{D}^\alpha u_x-u_N \mathcal{D}^\alpha (u_N)_x &= (u-u_N)\mathcal D^\alpha u_x - (u-u_N)\mathcal{D}^\alpha(u-u_N)_x + u\mathcal{D}^\alpha(u-u_N)_x,
    \end{aligned}
\end{equation}
where $\tilde{f''}_u=-\frac{3}{2}\gamma$ is the Taylor remainder.
Now, each term $\mathcal{I}_k$ is estimated using projection properties \eqref{Prop_Nr}, the \textit{a priori} assumption \eqref{prioriass} and \eqref{ass1uu_h2}, Cauchy-Schwarz inequality, smoothness of the exact solution $u$, and Young's inequality. Consequently for $\mathcal{I}_1$, we have 
\begin{equation*}
    \begin{aligned}
        \mathcal{I}_1 &\leq \|f'(u)-P_Nf'(u)\|_{\infty}\bigl(\|\mathcal E_N\|\|(\mathcal E_N)_x\|+\|P_Nu-u\|\|(\mathcal E_N)_x\|\bigr)\\&\quad+ \bigl(P_Nf'(u)(\mathcal E_N-(P_Nu-u)),(\mathcal E_N)_x\bigr) + C \|u-U_N\|_{\infty}\bigl(\|\mathcal E_N\|\|(\mathcal E_N)_x\|+\|P_Nu-u\|\|(\mathcal E_N)_x\|\bigr)\\&
        \leq CN^{-2r} + C\|\mathcal E_N\|^2,
    \end{aligned}
\end{equation*}
where we used the inverse inequality $\|(\mathcal{E}_N)_x\| \leq CN\|\mathcal{E}_N\|$ for $\mathcal{E}_N \in V_N$. Similar arguments imply
\begin{equation*}
    \begin{aligned}
        \mathcal{I}_2 &\leq C \|u-P_Nu\|_{\infty}\bigl(\|\mathcal E_N\|\|(\mathcal D^\alpha\mathcal E_N)_x\|+\|P_Nu-u\|\|(\mathcal D^\alpha\mathcal E_N)_x\|\bigr) +2\kappa_2 \bigl(P_Nu\mathcal{D}^{\frac{\alpha}{2}}\mathcal E_N, (\mathcal{D}^{\frac{\alpha}{2}}\mathcal E_N)_x\bigr)  \\& \quad -2\kappa_2 \bigl(P_Nu(P_Nu-u), (\mathcal{D}^{\alpha}\mathcal E_N)_x\bigr)+ C \|u-u_N\|_{\infty}\bigl(\|\mathcal E_N\|\|(\mathcal D^\alpha\mathcal E_N)_x\|+\|P_Nu-u\|\|(\mathcal D^\alpha\mathcal E_N)_x\|\bigr)\\& 
        \leq CN^{-2r} + C\|\mathcal E_N\|^2 + C\|\mathcal D^{\frac{\alpha}{2}}\mathcal E_N\|^2.
    \end{aligned}
\end{equation*}
\begin{equation*}
\begin{aligned}\label{eq:est_phii3}
        \mathcal{I}_3 & \leq \|\mathcal D^\alpha u_x\|_{\infty}\|u-u_N\|\|\mathcal E_N\| - \bigl((u-u_N)\mathcal{D}^{\alpha}(u-u_N)_x, \mathcal E_N\bigr) \\& \qquad
        +\bigl((u-P_Nu)\mathcal{D}^{\alpha}(u-u_N)_x, \mathcal E_N\bigr) +\bigl(P_Nu\mathcal{D}^\alpha(u-u_N)_x, \mathcal E_N\bigr) 
        \\&\leq 
        \|\mathcal D^\alpha u_x\|_{\infty}\bigl(\|\mathcal E_N\|^2+\|P_Nu-u\|\|\mathcal E_N\|\bigr) \\& \qquad+ C \|u-u_N\|_{\infty}\bigl(\|(\mathcal D^\alpha\mathcal E_N)_x\|\|\mathcal E_N\|+\|\mathcal{D}^{\alpha}(P_Nu-u)_x\|\|\mathcal E_N\|\bigr)\\& \qquad + C \|u-P_Nu\|_{\infty}\bigl(\|(\mathcal D^\alpha\mathcal E_N)_x\|\|\mathcal E_N\|+\|\mathcal{D}^{\alpha}(P_Nu-u)_x\|\|\mathcal E_N\|\bigr) -\bigl(P_Nu\mathcal{D}^{\frac{\alpha}{2}}\mathcal E_N, (\mathcal{D}^{\frac{\alpha}{2}}\mathcal E_N)_x\bigr) \\& \qquad + \bigl(P_Nu(P_Nu-u), (\mathcal{D}^{\alpha}\mathcal E_N)_x\bigr)
        \\&\leq 
        \|\mathcal D^\alpha u_x\|_{\infty}\bigl(\|\mathcal E_N\|^2+\|P_Nu-u\|\|\mathcal E_N\|\bigr) \\& \qquad+ 2C N^{-(\alpha+1)}\bigl(CN^{(\alpha+1)}\|\mathcal E_N\|^2+CN^{-r+(\alpha+1)}\|u\|_r\|\mathcal E_N\|\bigr) 
        \\& \leq CN^{-2r} + C\|\mathcal E_N\|^2.
    \end{aligned}
\end{equation*}
Substituting estimates on $\mathcal I_1$, $\mathcal I_2$, and $\mathcal I_3$  into \eqref{eq:est_phi} produces the differential inequality
\begin{equation}\label{eq:est_phi2}
    \begin{aligned}
        \frac{1}{2}\frac{d}{dt}\left(\|\mathcal{E}_{N}\|^2 + \|\mathcal{D}^{\frac{\alpha}{2}}\mathcal{E}_N\|^2\right) \leq CN^{-2r} + C\|\mathcal E_N\|^2 + C\|\mathcal D^{\frac{\alpha}{2}}\mathcal E_N\|^2.
    \end{aligned}
\end{equation}
An application of Gr\"onwall's inequality to \eqref{eq:est_phi2} over $[0,T]$ gives
\begin{equation*}
    \|\mathcal{E}_N(t)\|^2 + \|\mathcal{D}^{\alpha/2} \mathcal{E}_N(t)\|^2 \leq e^{CT}N^{-2r},  \qquad \forall t \in [0,T].
\end{equation*}
The triangle inequality and projection estimate \eqref{Prop_Nr} yield
\begin{align*}
    \|u - u_N\|_{L^2} \leq CN^{-r},
\end{align*}
which establishes \eqref{eq:L2_error_bound} for all $t\in [0,T]$. When \(u\) is analytic, using the analytic projection estimate \eqref{alpha_est2} in the estimates \eqref{eq:est_phi}-\eqref{eq:est_phii3} and similar analysis yields
\[
\frac{1}{2}\,\frac{\mathrm{d}}{\mathrm{d}t}\!\left(\|\mathcal{E}_{N}\|^2+\|\mathcal{D}^{\alpha/2}\mathcal{E}_{N}\|^2\right)\le C\mathrm{e}^{-2cN}+C\|\mathcal{E}_{N}\|^2+C\|\mathcal{D}^{\alpha/2}\mathcal{E}_{N}\|^2.
\]
An application of Gr\"onwall’s inequality over \([0,T]\) then shows
\[
\|\mathcal{E}_{N}(t)\|^2+\|\mathcal{D}^{\alpha/2}\mathcal{E}_{N}(t)\|^2 \le \mathrm{e}^{CT}\,\mathrm{e}^{-2cN},\qquad\forall\,t\in[0,T].
\]
Finally, by the triangle inequality and the analytic projection estimate (cf. Lemma \ref{Prop_Nr}) we obtain
\[
\|u(t)-u_{N}(t)\|_{L^2}
\le \|\mathcal{E}_{N}(t)\|_{L^2} + \|u(t)-P_{N}u(t)\|_{L^2}
\le C\mathrm{e}^{-cN},
\]
which establishes the exponential convergence claim and completes the proof.
\end{proof}
\begin{remark}\label{Ass_vald}
    The admissibility of the assumption \eqref{prioriass} is established via a continuity argument. For $r \geq \alpha + 2$ and $N > N_0$ sufficiently large, the constant $C = C(T, r, \|u\|_{C([0,T];H^r(I))})$ in the error estimate \eqref{eq:L2_error_bound}, which can be exactly determined by the final time $T$, satisfies $CN^{-r} < \frac{1}{2}N^{-(\alpha+3/2)}$. Define the critical time  
\[
t^* := \sup \left\{ \tau \in [0,T] : \|u(t) - u_N(t)\|_{L^2} \leq N^{-(\alpha+3/2)} \ \forall t \in [0,\tau] \right\}.
\]
By continuity, we would have  
\(
\|u(t^*) - u_N(t^*)\|_{L^2} = N^{-(\alpha+3/2)}.
\)
Now suppose, for contradiction, that $t^* < T$. Then the Theorem \ref{thm:spectral_accuracy} holds for $t \in [0,t^*]$, and it implies 
\[
\|u(t^*) - u_N(t^*)\|_{L^2} \leq CN^{-r} < \tfrac{1}{2}N^{-(\alpha+3/2)},
\]  
contradicting the equality above. Hence $t^* \geq T$, validating \eqref{prioriass} throughout $[0,T]$.
\end{remark}

\section{Numerical Experiments}
\label{sec5}

This section presents comprehensive numerical experiments for the fCH equation \eqref{eq:fch}.  The spatial discretization uses the FSG method \eqref{eq:fsg} with \(2N+1\) basis functions, while temporal integration employs a fourth-order \textit{Runge-Kutta} (RK4) method. The FSG scheme achieves spectral accuracy for smooth solutions (Theorem \ref{thm:spectral_accuracy}). For non-smooth solutions (e.g., peakons), we observe reduced convergence rates consistent with the limited solution regularity. All experiments use final time \(T \), with time steps $\Delta t \approx 0.001/k_{\max}$, where $k_{\max}$ is the maximum resolved wavenumber, to ensure temporal errors are negligible compared to spatial discretization errors. For the numerical experiments' purposes, we choose a convenient spatial domain according to the initial and boundary data.

% The semi-discrete FSG scheme \eqref{eq:fsg} yields the ODE system
% \[
% \frac{d\mathbf{U}}{dt} = \mathbf{g}(\mathbf{U}(t)), \quad \mathbf{U}(t) = [\hat{u}_{-N}(t), \dots, \hat{u}_N(t)]^T,
% \]
% solved using RK4
% \begin{align*}
% \mathbf{k}_1 &= \mathbf{g}(\mathbf{U}^n), \\
% \mathbf{k}_2 &= \mathbf{g}(\mathbf{U}^n + \frac{\Delta t}{2}\mathbf{k}_1), \\
% \mathbf{k}_3 &= \mathbf{g}(\mathbf{U}^n + \frac{\Delta t}{2}\mathbf{k}_2), \\
% \mathbf{k}_4 &= \mathbf{g}(\mathbf{U}^n + \Delta t\mathbf{k}_3), \\
% \mathbf{U}^{n+1} &= \mathbf{U}^n + \frac{\Delta t}{6}(\mathbf{k}_1 + 2\mathbf{k}_2 + 2\mathbf{k}_3 + \mathbf{k}_4).
% \end{align*}
Spectral accuracy is verified via \(L^2\)-error decay \(\|u - u_N\| \sim \mathcal{O}(e^{-cN})\) for smooth solutions. For non-smooth solutions, algebraic convergence \(\|u - u_N\| \sim \mathcal{O}(N^{-r})\) is expected, where \(r\) is actually a solution regularity.

\subsubsection{ Example 1: Smooth Traveling Wave}
\label{subsec:smooth_wave} This example validates the spectral accuracy of the FSG method \eqref{eq:fsg} for the fCH equation \eqref{eq:fch}. We consider the classical Camassa-Holm equation ($\alpha = 2$) with a smooth traveling wave solution \cite{xu2008local}. The exact solution is a smooth periodic traveling wave of the form $u(x,t) = \phi(x - ct)$, where the wave profile $\varphi(\xi)$ satisfies the ordinary differential equation
$$\varphi_{\xi\xi} = \varphi - \frac{3}{(\varphi - c)^2}$$
with initial conditions $\varphi(0) = 1$ and $\varphi_\xi(0) = 0$. The wave speed is $c = 3$, and the solution has a period $a \approx 6.4695$. The physical parameters are $(\kappa_1, \gamma, \kappa_2) = (0, 1, \frac{1}{3})$, corresponding to the standard Camassa-Holm nonlinearity \cite{xu2008local}.

The spatial domain is $[0, 10a]$ with periodic boundary conditions. The initial condition is $u_0(x) = \phi(x)$, and we solve until the final time $T = 1$. The FSG method \eqref{eq:fsg} is implemented with $N =  16, 32, 64, 128, 256$ Fourier modes. 

 Figure \ref{fig:smooth_results} shows the numerical solution compared with the exact solution at $T = 1$ for $N = 256$ modes and demonstrates exponential convergence, where both $L^2$ and $L^\infty$ errors decay exponentially with increasing $N$, as evidenced by the comparison with the exponential reference curve $\mathcal O(e^{-cN})$. The rapid error decay, with convergence orders increasing exponentially as shown in Table \ref{tab:smoothconvergence}, confirms the exponential accuracy of the method. 

\begin{figure}[h!]
\centering
\includegraphics[width=\textwidth, height=7cm]{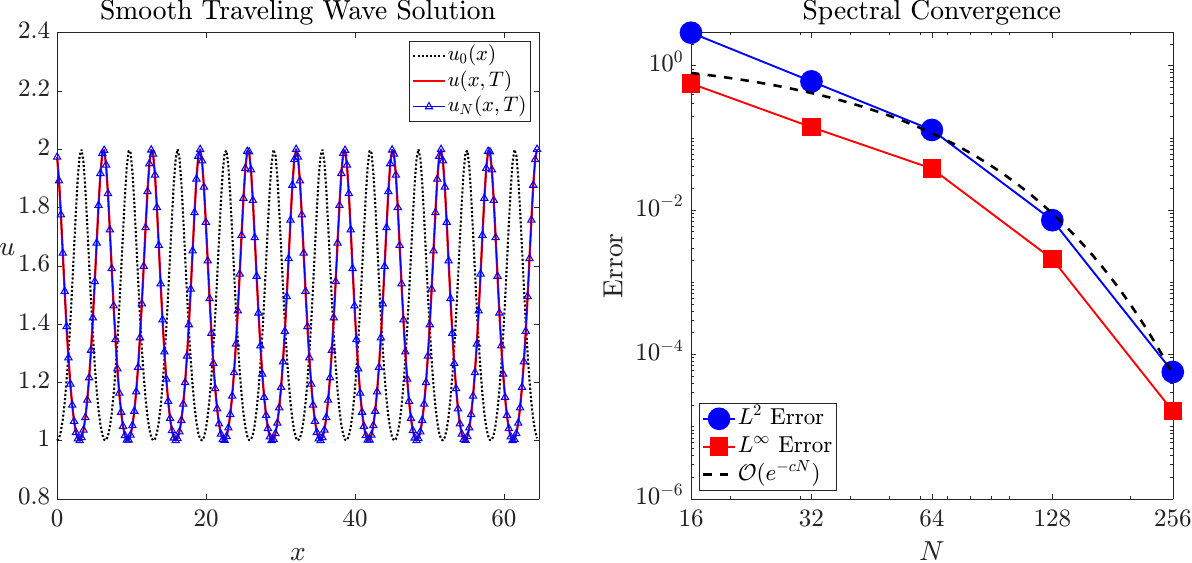}
\caption{
Numerical results for smooth traveling wave solution: (Left) Initial condition, exact solution, and numerical solution at $T=1$ with $N=256$; (Right) Exponential convergence of $L^2$ and $L^\infty$ errors with increasing $N$.
}
\label{fig:smooth_results}
\end{figure}

\begin{table}[h!]
\centering
\begin{tabular}{||c|c|c|c|c||}
\hline
$N$ & $L^2$--Error & Order & $L^\infty$--Error & Order \\
\hline\hline
$16$  & $2.853$             & --- & $5.643\times 10^{-1}$ & --- \\
$32$  & $6.041\times 10^{-1}$ & $2.24$ & $1.403\times 10^{-1}$ & $2.01$ \\
$64$  & $1.284\times 10^{-1}$ & $2.23$ & $3.749\times 10^{-2}$ & $1.90$ \\
$128$ & $7.225\times 10^{-3}$ & $4.15$ & $2.093\times 10^{-3}$ & $4.16$ \\
$256$ & $5.712\times 10^{-5}$ & $6.98$ & $1.640\times 10^{-5}$ & $7.00$ \\
\hline
\end{tabular}
\caption{Convergence analysis for smooth traveling wave solution $u(x,t) = \varphi(x-ct)$ at $T=1$ and spatial nodes $N$.}
\label{tab:smoothconvergence}
\end{table}

\subsubsection*{Example 2: Classical Camassa-Holm Peakon (\(\alpha = 2\))}
\label{ex:peakon_alpha2} The parameter set $(\kappa_1, \gamma, \kappa_2, \alpha) = (0, 1, \frac{1}{3}, 2)$ reduces the fCH equation \eqref{eq:fch} to the classical Camassa-Holm (CH) equation. This model admits {\it peakon} solutions - weakly singular traveling waves with discontinuous first derivatives that model shallow water wave breaking \cite{holden2006convergence}. 

 We consider the single peakon solution:
\begin{equation}\label{eqn:peaksoln}
u_0(x) = ce^{-|x - x_0|}, \quad c = 1, \quad x_0 = 25, \quad x \in [0, 50].
\end{equation}
The exact solution propagates rigidly: $u(x,t) = ce^{-|x - ct - x_0|}$. The discontinuity in $u_x$ at the peak $x = ct + x_0$ restricts the solution to $H^1(I)$, limiting convergence rates. Figure \ref{fig:ex1_profile} shows peakon propagation using $N = 1024$ modes and time $T=10$. Despite the derivative discontinuity, the FSG solution maintains the characteristic peak structure without spurious oscillations. Table \ref{tab:ex1_conv} quantifies convergence in the smooth region $|x - (ct + x_0)| > 1$. The observed $\mathcal{O}(N^{-1})$ convergence aligns with theoretical expectations for $H^1$ solutions. 

\begin{table}[htbp]
\centering
\begin{tabular}{||c|c|c|c|c||}
\hline
$N$ & $L^2$--Error & Order & $L^\infty$--Error & Order\\
\hline
\hline
512  & $4.2900\times 10^{-2}$ & --- & $3.9298\times 10^{-2}$ & ---\\
1024 & $2.0598\times 10^{-2}$ & 1.06 & $1.7498\times 10^{-2}$ & 1.17\\
2048 & $1.0166\times 10^{-2}$ & 1.02 & $8.2880\times 10^{-3}$ & 1.08\\
4096 & $5.0467\times 10^{-3}$ & 1.01 & $4.0551\times 10^{-3}$ & 1.03\\
8192 & $2.4893\times 10^{-3}$ & 1.02 & $1.9910\times 10^{-3}$ & 1.03\\
\hline
\end{tabular}
\caption{Spectral convergence for a smooth solution.  The table lists the $L^2$ and $L^\infty$ errors and the observed convergence order as the number of spatial nodes $N$ increases.}
\label{tab:ex1_conv}
\end{table}

\begin{figure}[h!]
\centering
\includegraphics[width=0.8\linewidth, height=7cm]{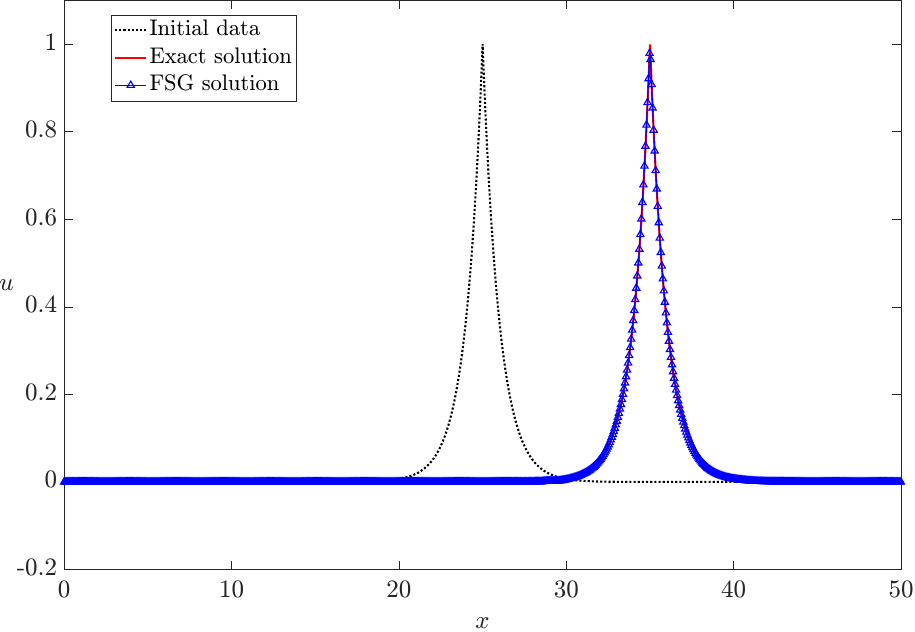}
\caption{Peakon propagation at \(t = 10\).}
\label{fig:ex1_profile}
\end{figure}

\subsubsection*{Example 3. Two and three-peakon interaction of Camassa-Holm equation} We consider the parameters $(\kappa_1,\gamma,\kappa_2,\alpha) = (0,1,\frac{1}{3},2)$ in \eqref{eq:fsg} with periodic boundary conditions on the domain $[0, L]$ where $L = 30$. The initial conditions consist of superposed peakon solutions of the form
\begin{equation}
\varphi_i(x) = \frac{c_i}{\cosh(L/2)} \min \left\{ \cosh \left( \min \left| x - x_i \right| \right), \cosh \left( L - \min \left| x - x_i \right| \right) \right\},
\end{equation}
periodized over the domain. For the two-peakon interaction with the initial condition
$$u_0(x) = \varphi_1(x) + \varphi_2(x),$$
we take parameters $(c_1,c_2)= (2,1)$ and $(x_1,x_2) = (-5,5)$
and for the three-peakon interaction with the initial condition
$$u_0(x) = \varphi_1(x) + \varphi_2(x) + \varphi_3(x),$$
 we take parameters $(c_1,c_2,c_3)= (2,1,0.8)$ and $(x_1,x_2,x_3) = (-5,-3,-1)$.
Numerical solutions are computed using the scheme \eqref{eq:fsg} with $N = 1024$ spatial points, time step $\Delta t =0.5\times 10^{-2}$, and RK4 time integration. Solutions are shown at times $t = \{0, 5, 12, 20\}$ for the two-peakon case and $t = \{0, 1, 2, 3\}$ for the three-peakon case.

Figures \ref{fig:2peakon_interaction} and \ref{fig:3peakon_interaction} show the evolution of the two peakon and three peakon solutions, respectively. In both cases, we observe that the peakon profiles maintain their shape during propagation and interact cleanly during collisions. The numerical scheme resolves the moving peaks and their nonlinear interactions without introducing significant dissipation or spurious oscillations, demonstrating the effectiveness of the spectral approach for solving the fCH equation \eqref{eq:fch} with non-smooth solutions.

\begin{figure}
\centering
\includegraphics[width=\linewidth, height=8cm]{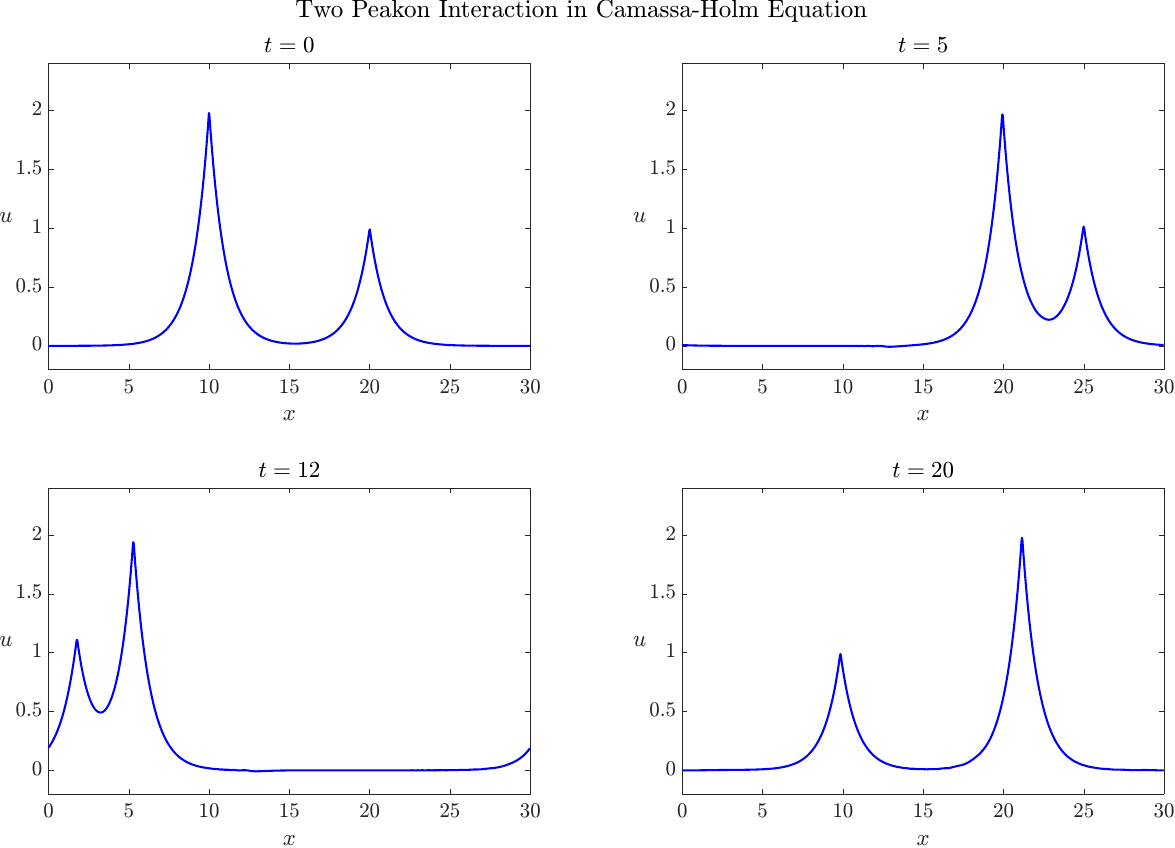}
\caption{Two-peakon interactions: $c_1 = 2$, $c_2 = 1$ at $t = 0, 5, 12, 20$ and nodes $N=1024$.}
\label{fig:2peakon_interaction}
\end{figure}

\begin{figure}
\centering
\includegraphics[width=\linewidth,height=8cm]{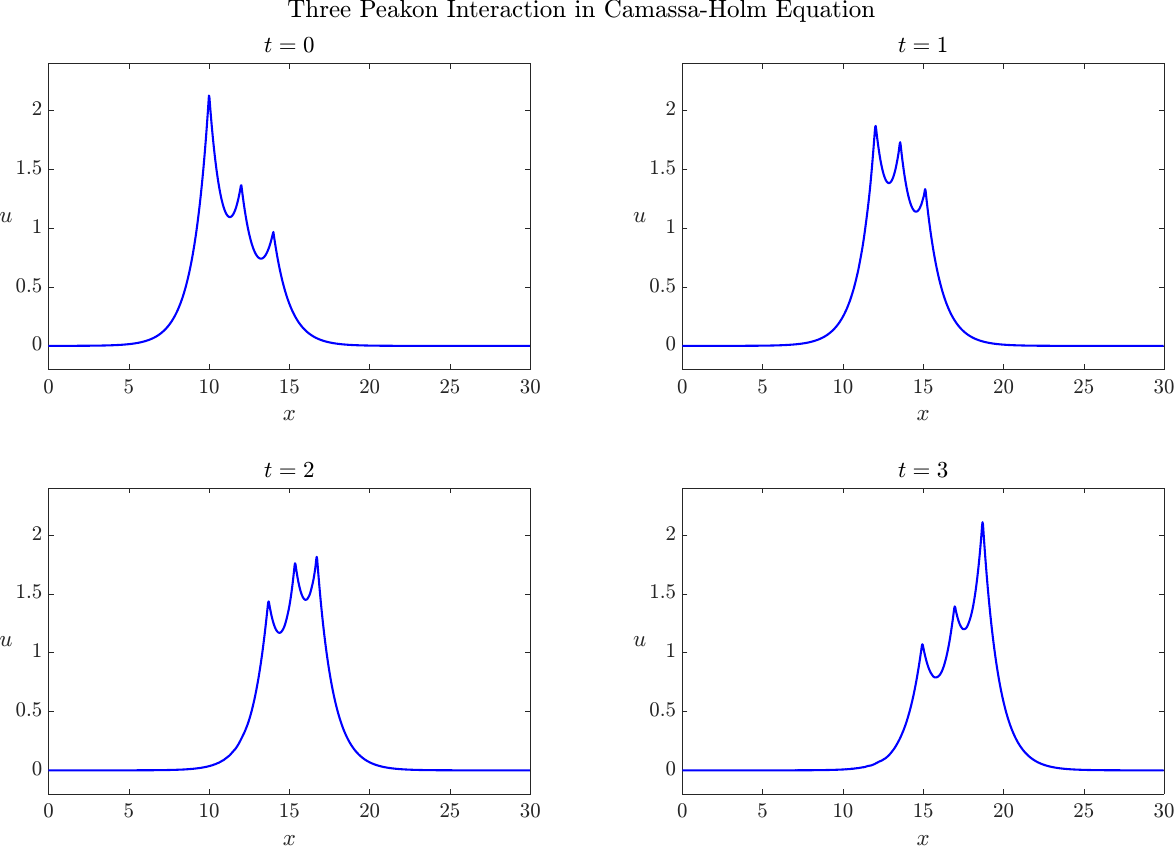}
\caption{Three-peakon interaction: $c_1 = 2$, $c_2 = 1$, $c_3 = 0.8$ at $t = 0, 1, 2, 3$, and nodes $N=1024$.}
\label{fig:3peakon_interaction}
\end{figure}

%%%%%%%%%%%

\subsubsection*{Example 4: Fractional Dispersion Effects ($1 \leq \alpha \leq 2$)}\label{exm:fracexamfch}
This experiment probes dispersion-modulated dynamics in the fCH equation \eqref{eq:fch} using $\alpha \in \{1.0, 1.4, 1.7, 2.0\}$ with fixed physical parameters ($\kappa_1 = 0$, $\gamma = 1$, $\kappa_2 = \frac{1}{3}$).

Numerical parameters: $L = 50$, $T = 10$, $\Delta t = 10^{-4}$, $N = 4096$. Initial peakon $u_0(x) = e^{-|x-25|}$ evolves as shown in Figure \ref{fig:ex2_profiles}. The scheme robustly captures the progressive amplitude decay and waveform broadening as $\alpha$ decreases, structural preservation at $\alpha = 2$ (classical limit), and smooth dispersion-nonlinearity transitions across $\alpha$ regimes, confirming consistent performance across fractional orders without stability degradation.

\begin{figure}[h!]
\centering
\includegraphics[width=\linewidth, height=8cm]{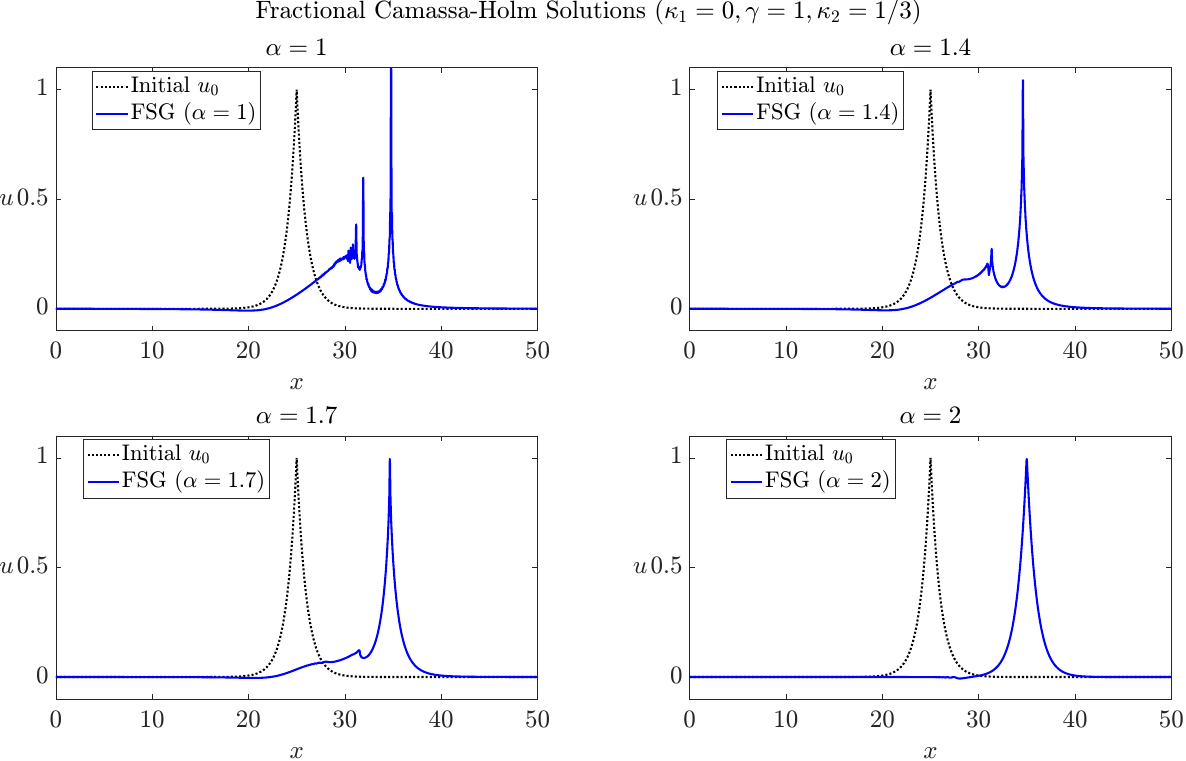}
\caption{Peakon evolution under fractional dispersion) at $T = 10$.}
\label{fig:ex2_profiles}
\end{figure}

\subsubsection*{Example 5: The fBBM Equation ($\alpha = 2$)}
\label{ex:bbm_classic}

This example considers the parameters $(\kappa_1,\gamma,\kappa_2,\alpha) = (1,\frac{1}{3},0,2)$ in the scheme \eqref{eq:fsg} which corresponds to the  classical Benjamin-Bona-Mahony (BBM) equation. The exact solitary wave solution is given by \cite{fabien2021high}
\[
u(x,t) = 3(c_s - 1) \sech\left[\frac{1}{2}\sqrt{\frac{c_s-1}{c_s}}(x - x_0 - c_s t)\right]^2
\]
with wave speed $c_s = 2$ and initial position $x_0 = -60$. The computational domain is $[-100, 100]$ with final simulation time $T = 50$, positioning the wave at $x = 40$.  

Numerical parameters include $N = 32$ to $512$ Fourier modes. Figure \ref{fig:bbm_convergence} (left) demonstrates excellent agreement between numerical and exact solutions, while Figure \ref{fig:bbm_convergence} (right) shows spectral convergence with exponential error decay. Table \ref{tab:bbm_convergence} demonstrates exponential convergence, with error reduction exceeding $\mathcal{O}(e^{-cN})$ as $N$ increases. This confirms the method's capability to accurately capture nonlinear wave phenomena while maintaining conservation properties.

\begin{figure}[h!]
\centering
\includegraphics[width=\linewidth, height=7cm]{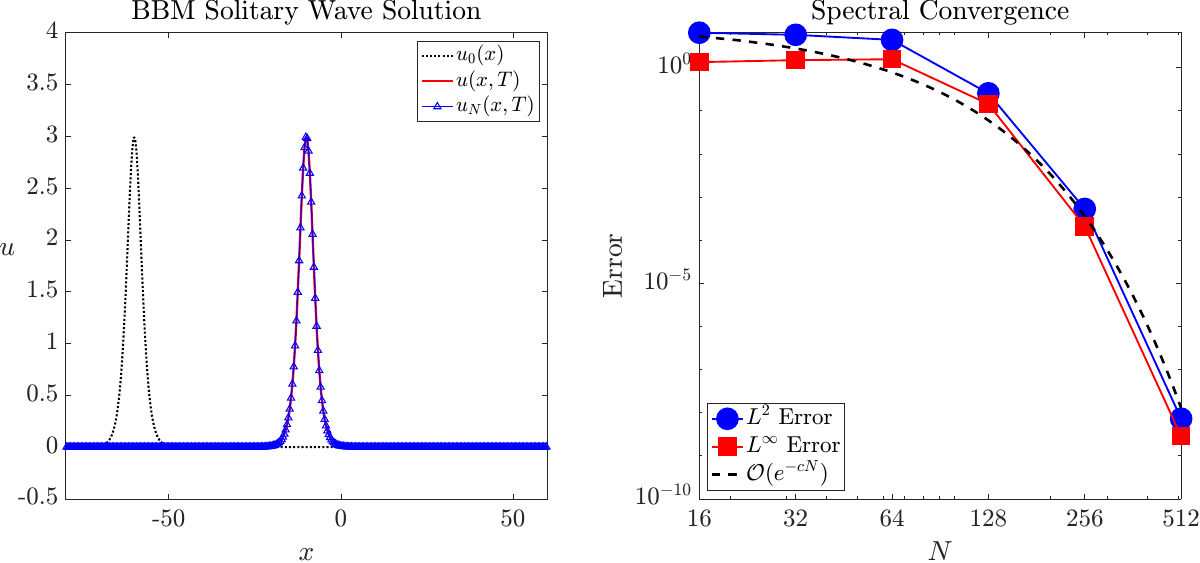}
\caption{(Left) Solitary wave solution at $T=50$ showing exact solution and numerical approximation ($N=2048$). (Right) Spectral convergence of $L^2$ and $L^\infty$ errors.}
\label{fig:bbm_convergence}
\end{figure}

\begin{table}[h]
\centering
\begin{tabular}{||c|c|c|c|c||}
\hline
$N$ & $L^2$--Error & Order & $L^\infty$--Error & Order \\
\hline
\hline
32 & $5.745 \times 10^{0}$ & --- & $1.486 \times 10^{0}$ & --- \\
64 & $4.387 \times 10^{0}$ & 0.39 & $1.562 \times 10^{0}$ & -0.07 \\
128 & $2.530 \times 10^{-1}$ & 4.12 & $1.423 \times 10^{-1}$ & 3.46 \\
256 & $5.282 \times 10^{-4}$ & 8.90 & $2.076 \times 10^{-4}$ & 9.42 \\
512 & $7.209 \times 10^{-9}$ & 16.16 & $2.846 \times 10^{-9}$ & 16.15 \\
\hline
\end{tabular}
\caption{Convergence analysis for BBM solitary wave ($\alpha=2$).}
\label{tab:bbm_convergence}
\end{table}

\section{Conclusion}\label{sec6}
We have developed and analyzed a Fourier spectral Galerkin method for the fractional Camassa–Holm equation with periodic boundary conditions.  By representing the fractional derivative exactly through Fourier multipliers, the semi–discrete scheme inherits the mass and energy conservation laws of the continuous problem.  In the simplified fractional Benjamin–Bona–Mahony setting, we established existence, uniqueness, and convergence of the numerical solution.  For general fractional Camassa–Holm equations, we derived optimal error bounds: when the initial data lies in $H^r(I)$ with $r\ge\alpha+2$, the $L^2$–error decays algebraically like $N^{-r}$; for analytic (hence smooth) solutions, the error decays exponentially like $\mathrm{e}^{-cN}$, confirming the spectral accuracy of the method. The numerical experiments corroborated these theoretical findings. Smooth travelling–wave solutions exhibited the expected rapid, nearly machine–precision convergence, while peakon solutions, whose regularity is limited to $H^1$, showed the anticipated algebraic decay.  

Future research will extend this framework in several directions.  One goal is to adapt the spectral method to higher–dimensional generalizations such as the Camassa–Holm–Kadomtsev–Petviashvili equation to model two–dimensional wave dynamics.  For non–periodic or irregular domains, mesh-free approaches such as radial basis function discretizations could provide a flexible alternative.  Additional avenues include designing structure–preserving time–integration schemes for long–time simulations, developing adaptive refinement strategies to resolve singularities, and incorporating variable–order fractional operators to model heterogeneous media.  Altogether, the present work lays a rigorous foundation for high–precision numerical exploration of a broad class of nonlinear wave phenomena.

\section*{Declaration} 
We have not used any data to conduct this work. The authors declare that this work does not have any conflicts of interest.
\bibliographystyle{abbrv}
\bibliography{main}

\end{document}